\documentclass[11pt,reqno]{amsart}
\usepackage[utf8]{inputenc}
\usepackage{a4wide}
\usepackage{xparse} %% necessary for \norm macro !
\usepackage{amssymb,amsthm,amsmath,eucal,mathrsfs,bbm,xspace}
\usepackage{enumitem}
\usepackage{tikz}
\usepackage{tikz-cd}
\usetikzlibrary{decorations.pathmorphing}

\usepackage[colorlinks=true,    % links will be colored instead of boxes
            linkcolor=blue,     % internal links (sections, theorems, etc.)
            citecolor=blue,     % bibliography citations
            urlcolor=blue]      % external URLs
           {hyperref}

\def\R{\mathbb{R}}

\def\C{\mathbb{C}}

\def\cL{\mathcal{L}}

\def\arg{\mathrm{arg}\,}
\def\ran{\mathrm{ran}\,}

\def\bel{\begin{equation}}
\def\eel{\end{equation}}

\def\beq{\begin{eqnarray*}}
\def\eeq{\end{eqnarray*}}

\def\eps{\varepsilon}

\def\om{\omega}

\def\la{\lambda}

\def\al{\alpha}

\def\Ga{\Gamma}
\def\si{\sigma}

\def\del{\delta}

\def\calE{\mathcal E}

\makeatletter
\newtheoremstyle{tagged}% name
  {3pt}% Space above
  {3pt}% Space below
  {\itshape}% Body font
  {}% Indent amount
  {\bfseries}% Theorem head font
  {}% Punctuation after theorem head
  { }% Space after theorem head
  {\thmname{#1}~(#3)}% Theorem head spec: name + tag
\makeatother

\newtheorem{proposition}{Proposition}[section]
\newtheorem{lemma}[proposition]{Lemma}
\newtheorem{remark}[proposition]{Remark}
\newtheorem{corollary}[proposition]{Corollary}

\newtheorem{theorem}[proposition]{Theorem}

\newtheorem{example}[proposition]{Example}

\theoremstyle{tagged}

\renewcommand\labelenumi{$(\roman{enumi})$}
\renewcommand\theenumi\labelenumi
\renewcommand{\Re}{\text{\rm Re}}

\numberwithin{equation}{section}

\newcounter{aufzi}

\NewDocumentCommand{\norm}{m o}{%
  \lVert #1 \rVert%
  \IfValueT{#2}{_{#2}}%
}
\NewDocumentCommand{\bignorm}{m o}{%
  \bigl\lVert #1 \bigr\rVert%
  \IfValueT{#2}{_{#2}}%
}
\NewDocumentCommand{\Bignorm}{m o}{%
  \Bigl\lVert #1 \Bigr\rVert%
  \IfValueT{#2}{_{#2}}%
}
  \newcommand{\dsp}{\displaystyle}

  \newcommand {\pihalbe} {\frac{\pi}{2}}
  \newcommand {\eins}{\mathbbm 1}
  \newcommand {\sfrac}[2] { {\,{}^{#1}\!\!/\!{}_{#2}\,}} 
  
  \newcommand {\NN} {\mathbb N}
  \newcommand {\ZZ} {\mathbb Z}
  \newcommand {\RR} {\mathbb R}

  \newcommand {\THUT} {\widehat{T}}
  \newcommand {\UHUT} {\widehat{U}}
  \newcommand {\calO} {\mathscr O}
  \DeclareMathOperator {\strip} {St}
  \DeclareMathOperator {\Sect} {\ensuremath{\Sigma}}
  \DeclareMathOperator {\bdd} {{\mathcal L}}
  
  \newcommand{\MG}{G}  % = Madani-Gamma !!
  \newcommand{\HP}{\mathbb H}  % HalfPlane
  \newcommand{\KR}{\mathbb K}
  \newcommand{\UMD}{\text{\upshape UMD}\xspace}

  \title{$H^\infty$--functional calculus for generators of
    semigroups that admit lower bounds} \author{Bernhard H. Haak}
  \address{\it Université de Bordeaux\\Institut de Mathématiques de
    Bordeaux\\351 cours de la Libération\\F -- 33405 Talence\\France}
  \email{bernhard.haak@math.u-bordeaux.fr}

\author{Peer Chr. Kunstmann}
\address{\it Karlsruhe Institute of Technology (KIT),
Institute for Analysis\\
Englerstr. 2, D -- 76128 Karlsruhe\\Germany}
\email{peer.kunstmann@kit.edu}

\keywords{$H^\infty$-functional calculus, $C_0$-semigroups, dilations,
  Banach spaces (especially UMD spaces), operator theory}
\subjclass{Primary: 47A60, 47D06;
  Secondary: 46B20, 47A10, 47B40, 47A20}

\date{}

\begin{document}
\maketitle

\begin{abstract}
  We study $C_0$-semigroups on \UMD Banach spaces under the assumption
  that a single semigroup operator admits a lower bound.  We establish
  boundedness of $H^\infty$-functional calculi for the negative
  generator of such semigroups.  Our approach is based on a dilation
  argument: combining a recent construction due to Madani with
  transference results for groups on \UMD spaces, we embed the
  semigroup into a $C_0$-group on a larger space and transfer
  functional calculus estimates back to the original generator.  As a
  byproduct, we obtain quantitative exponential lower bounds for the
  semigroup.  We also show that equivalences due to Batty and Geyer,
  valid in Hilbert spaces, fail in the general Banach space setting.
\end{abstract}
  
\section{Introduction}
The bounded $H^\infty$-functional calculus for negative generators of
$C_0$-semigroups on Banach spaces has become a central tool in modern
analysis, with applications ranging from evolution equations to
harmonic analysis.  For a strongly continuous group
$(S(t)) = (e^{-itA})$ on a Hilbert space, Boyadzhiev and de~Laubenfels
\cite{B-dL} showed that $A$ has a bounded $H^\infty$-functional
calculus on a horizontal strip whose height is determined by the growth
bounds of the group. This result has been extended to bounded groups
on \UMD Banach spaces by Hieber and Prüss \cite{Hieber-Pruess} using
transference principles.  They showed that, for generators $A$ of
bounded $C_0$-groups, the operator $iA$ admits a bounded
$H^\infty$-calculus on a suitable double-sector symmetric about the
real axis. This was later extended by Haase \cite{Haase:transference}
to strongly continuous groups on \UMD spaces with arbitrary growth
bounds, at the expense of obtaining a bounded $H^\infty$ functional
calculus on ``Venturi regions'' which are unions of double sectors and
strips.

\medskip

A natural question is how to obtain such results for semigroups that
are not groups.  A classical approach consists in dilating a semigroup
into a group.  For instance, the Sz.-Nagy dilation theorem shows that
a contraction semigroup on a Hilbert space admits a unitary dilation,
i.e., it can be realized as the compression of a unitary group acting
on a larger Hilbert space.  Fröhlich and Weis \cite{F-W} constructed a
dilation of $C_0$-semigroups on Banach spaces with finite cotype by
means of square function estimates, thereby providing another route to
functional calculus results. Recently, Madani \cite{Madani} showed
that a single lower bound for a bounded operator on a reflexive Banach
space suffices to construct an invertible dilation on a reflexive
space.  The purpose of this paper is to combine Madani's dilation
technique with the findings of Haase in order to derive boundedness of
the $H^\infty$-functional calculus for negative generators of
$C_0$-semigroups from a simple lower estimate at a single time.  More
precisely, we show that if a semigroup $(T(t))$ in a \UMD space $X$
satisfies
\[
  \forall x\in X:\ \norm{ T(t_0)x } \ge c \, \norm{ x }
\]
for some $t_0, c > 0$, then its negative generator $A$ admits a
bounded $H^\infty$-calculus on regions that are unions of sectors and
half-planes.  The argument proceeds by constructing a dilation of the
semigroup to a $C_0$-group on a suitable \UMD space $Y$, using
Madani's approach. The functional calculus for the generator of the
dilation group in $Y$ that we have by Haase's result is then (partly)
transferred back to the negative generator $A$ of the initial
semigroup.

We also record a quantitative lower estimate for semigroups that
follows from the dilation construction and may be of independent
interest.

\subsection*{Notation and Main result}
We write $\bdd(X)$ for the space of bounded operators on a Banach
space $X$.  For a fixed operator $T \in \bdd(X)$ we write
$\{ T \}'$ for its commutant, i.e. the algebra of all $U \in \bdd(X)$
such that $UT = T U$.  For any nontrivial complex domain $\calO$ we
denote by $H^\infty(\calO)$ the space of bounded holomorphic functions on
$\calO$. For $\si\in(0,\pi)$ we denote by $\Sect_\si$ the open sector
\[
 \Sect_\si=\{z\in\C\setminus\{0\}: |\arg z|<\si\}.
\]
We call an operator $A$ in a Banach space $X$ sectorial, if $A$ is densely
defined with dense range and there exists $\om\in(0,\pi)$ such that the
spectrum $\si(A)$ of $A$ is contained in $\overline{\Sect_\om}$ and we have,
for any $\si\in(\om,\pi)$,
\[
\sup\{\norm{\la R(\la,A)}_{\bdd(X)} : \la\in\C\setminus\overline{\Sect_\si}
\,\}<\infty.
\]
In this case, we say that $A$ has a bounded $H^\infty(\Sect_\si)$-calculus
for a fixed $\si>\om$ if there exists $C>0$ such that we have
\[
\norm{ f(A) }_{\bdd(X)} \le C \, \norm{ f }_{H^\infty(\Sect_\si)}.
\]
for all $f\in H^\infty(\Sect_\si)$ with
\[
\sup_{z\in\Sect_\si}\frac{|z|^\eps}{1+|z|^{2\eps}}|f(z)|<\infty
\]
for some $\eps>0$, where the operator $f(A)$ for such $f$ is defined by
the absolutely convergent integral
\[
f(A)=\frac1{2\pi i}\int_{\partial\Sect_\eta} f(\la) R(\la,A)\,d\la
\]
with $\eta\in(\om,\si)$. Here, $\partial\Sect_\eta$ is parametrized such that
$\Sect_\om$ lies to the left. If $A$ has a bounded $H^\infty(\Sect_\si)$-calculus
in this sense then the map $f\mapsto f(A)$ has indeed a unique extension to an
algebra homomorphism $H^\infty(\Sect_\si)\to\bdd(X)$. For this and further
details we refer to \cite{Haase:buch}.

%We say that $A$ has a bounded $H^\infty(\calO)$ functional
%calculus, if there exist a bounded algebra homomorphism
%$\Psi: H^\infty(\calO) \to \bdd(X)$. In particular, writing $f(A)$ for
%$\Psi(f)$, there is some constant $C>0$ such that
%\[
%  \norm{ f(A) }_{\bdd(X)} \le C \, \norm{ f }_{H^\infty(\calO)}.
%\]
We recall that a Banach space $X$ has the \UMD\ property if the
Hilbert transform extends to a bounded operator on $L^p(\RR; X)$ for
some (equivalently, all) $1<p<\infty$.  We refer to \cite{fab4:band1}
for detailed discussion and references.

\begin{theorem}\label{thm:HP-transfer}
  Let $X$ be a \UMD space and let $(T(t))_{t\ge 0}$ be a
  $C_0$-semigroup on $X$ with generator $-A$ and with growth bound
  $\norm{T(t)} \le M \, e^{\omega t}$, $t\ge0$, where $M\ge1$ and $\omega\in\R$.
  Assume %that its generator ${-}A$ is injective and
  that the semigroup satisfies the lower bound
\begin{align}\label{eq:lower-bd-t0}
  \norm{ T(t_0)x }_X \ge  c \, \norm{x}, \quad(x\in X)
\end{align}
for some $t_0, c>0$.  Then, for any $\sigma>\pihalbe$ and any
$\theta > \omega$, $A+\theta$ admits a bounded
$ H^\infty( \Sect_\si) $ calculus.
\end{theorem}

In the final part of this section, we consider a $C_0$-semigroup
$(T(t))_{t\ge0}$ in an arbitrary Banach space $X$ and assume that
$t_0,c>0$ are such that the lower bound \eqref{eq:lower-bd-t0} holds.
As a consequence of Madani's construction that we review in
Section~\ref{sec:madani}, we can then show the following.

\begin{proposition}\label{prop:sg-lower-bound}
  Let $X$ be a Banach space and $(T(t))_{t\ge0}$ a $C_0$-semigroup on
  $X$.  Assume $t_0,c>0$ are such that the lower bound
  \eqref{eq:lower-bd-t0} holds. Then we have the following:
  \begin{enumerate}[label=\alph*)]
  \item\label{item:sg-lower-a}
  For all $\alpha>1$ there exists $m>0$ such that
  \[
    \norm{ T(t) x } \ge  m\,e^{\nu t} \norm{x},\quad x\in X, t>0.
  \]
  where  $\nu = \tfrac1{t_0}\ln(\frac{c}{\alpha})$.
  \item\label{item:sg-lower-b}
  There exists a larger Banach space $Y\supseteq X$, containing $X$ as a
  closed subspace, and a $C_0$-group $(S(t))_{t\in\R}$ on $Y$ such that
  $T(t)=S(t)$ on $Y$ for all $t\ge0$.
  Moreover, one can achieve that $Y$ is reflexive if $X$
  is reflexive and that $Y$ is \UMD if $X$ is \UMD.
%  on $Y$$C_0$-semigroup $\big(\Phi(T(t))\big)_{t\ge0}$ in $Y$ obtained by
%  application of Theorem~\ref{thm:Madani} and Corollary~\ref{cor:Madani}
%  to $T:=T(t_0)$ extends to a $C_0$-group $(S(t))_{t\in\R}$ in $Y$.
   \end{enumerate}
\end{proposition}

\noindent %As before, in part~\ref{item:sg-lower-b}, if $X$ is reflexive or \UMD
%then $Y$ is reflexive or \UMD, respectively.
Part~\ref{item:sg-lower-a} seems not to have been
clearly stated in the literature.  Considering, in this situation,
  \[
   \gamma(t):=\big(\inf\{\norm{T(t)x}: x\in X, \norm{x}=1\}\big)^{-1},\quad t>0,
  \]
  one has $\gamma(t+s)\le\gamma(t)\gamma(s)$ for $s,t>0$ by the semigroup property
  which implies exponential boundedness of $\gamma(t)$ for
  $t\to\infty$. However, boundedness of $\gamma(t)$ as $t\to0+$ is
  less clear, as there exist semigroups $(V(t))_{t>0}$ which are
  strongly continuous on $(0,\infty)$ and for which $\norm{V(t)}$
  grows arbitrarily fast as $t\to 0+$, see
  e.g. Example~\ref{example:sg-growth} below.

  \bigskip \noindent If $X$ is a Hilbert space then one can say more than is stated
  above. In \cite[Theorem~7.3]{B-G} the following is shown.
  
\begin{theorem}[Batty-Geyer]\label{thm:BG}
  Let $T(t)$ be a strongly continuous semigroup on a Hilbert space
  $X$ with generator ${-}A$. Then the following are equivalent.
  \begin{enumerate}[label=\alph*)]
  \item\label{item:BG-a} $T$ satisfies lower bounds, i.e.
    $\norm{ T(t) x } \ge c(t) \norm{x}$ for some strictly positive
    function $c$.
  \item\label{item:BG-b} There exists a left inverse semigroup $S$ of
    $T$ on $X$.
  \item\label{item:BG-c} There exists a $C_0$-group extension $(S(t))_{t\in\R}$
    of $(T(t))_{t\ge0}$ on a larger Hilbert space $Y \supseteq X$ that contains
    $X$ as a closed subspace.
  \end{enumerate}
\end{theorem}

\noindent In Proposition~\ref{prop:sg-lower-bound} we have seen that a
single lower bound implies \ref{item:BG-a} and \ref{item:BG-c} on
arbitrary Banach spaces. This raises the question whether the
implication \ref{item:BG-a} $\Rightarrow$ \ref{item:BG-b} remains
valid on general Banach spaces, or at least within well-behaved
classes such as \UMD spaces.  This is however not true, as we will show
in Example~\ref{example:BG} below.

The paper is organized as follows: In Section~\ref{sec:madani} we review
Madani's construction of a dilation. In Section~\ref{sec:semigroups} we prove
Proposition~\ref{prop:sg-lower-bound} and give examples, and in
Section~\ref{sec:proof-main} we prove Theorem~\ref{thm:HP-transfer} and
give a more precise result as a corollary to the proof.

\section{Madani's construction of the dilation}\label{sec:madani}

\noindent We review Madani's construction of the dilation from \cite{Madani}.
Since this lies at the heart of all that is to follow we give full details.

\subsection{Preliminaries}
We start with  with some basic facts that are easily verified.

\noindent Let $X$ be a Banach space and $p\in(1,\infty)$. For any
operator $U\in\bdd(X)$ we denote by $\UHUT\in\bdd(\ell_p(\NN; X)$ the
operator given by $\UHUT(x_n)_n:=(Ux_n)_n$. We note that
$\norm{\UHUT}=\norm{U}$.

\noindent If $(T(t))_{t\ge0}$ is a $C_0$-semigroup in $X$ then
$(\THUT(t))_{t\ge0}$ is a $C_0$-semigroup in $\ell_p(\NN; X)$. Strong
continuity follows from strong continuity on finite sequences and
\[
\sup_{t\in[0,1]}\norm{\THUT(t)}=\sup_{t\in[0,1]}\norm{T(t)}<\infty.
\]
Now suppose that $F$ is a closed subspace of $\ell_p(\NN; X)$. We denote
by $Y_F$ the quotient space $\ell_p(\NN; X)/F$ and write $[x]=x+F$ for the
co-class of $x\in\ell_p(\NN; X)$.

\begin{lemma}\label{lem:elem}
  If $F$ is invariant under $\UHUT$, where $U\in\bdd(X)$, then $
  \Phi_F(U)[x]=[\UHUT x]$ defines an operator $\Phi_F(U)\in\bdd(Y_F)$
  and we have $\norm{\Phi_F(U)}\le\norm{U}$.

  If $(T(t))_{t\ge0}$ is a $C_0$-semigroup in $X$ and $F$ is invariant
  under all operators $\THUT(t)$, $t\ge0$, then $\big(\Phi_F(T(t))\big)_{t\ge0}$
  is a $C_0$-semigroup in $Y_F$.
\end{lemma}
\begin{proof}
The assertion on $\Phi_F(U)$ is standard.
The semigroup property of $\big(\Phi_F(T(t))\big)$ is inherited from
$(\THUT(t))$ and strong continuity
follows from
\[
\norm{\Phi_F(T(t))[x]-[x]}_{Y_F} =\norm{[\THUT(t)x-x]}_{Y_F}
\le\norm{\THUT(t)x-x}_{\ell_p(\NN ;X)} \to 0\ (t\to0+)
\]
for all $x\in \ell_p(\NN; X)$.
\end{proof}

We shall also need the right shift $R$ on $\ell_p(\NN; X)$ defined by
$R(x_1,x_2,x_3,\ldots)=(0,x_1,x_2,x_3,\ldots)$. We note that $R$ is an
isometry onto its range and that $R$ commutes with every operator
$\UHUT$, $U\in\bdd(X)$.

\subsection{Madani's dilation}
We assume that $T\in\bdd(X)$ satisfies the lower bound
\begin{equation}    \label{eq:lower-bound}
    \norm{ Tx } \ge c \, \norm{x}, \qquad x \in X,
\end{equation}
where $c>0$. Madani constructed a dilation of $T$ to an isomorphism $S$
on a Banach space $Y$ that contains (an isomorphic copy of) $X$ as a
closed subspace (\cite{Madani}). Moreover, $Y$ is reflexive if $X$ is
reflexive, which was the main motivation for the construction.

We denote by $j:X\to \ell_p(\NN;X)$ the injection $j(x)=(x,0,0,\ldots)$.
The idea is to construct a suitable $\MG\in\bdd(\ell^p(\NN;X))$ with
closed range and commuting with $\THUT$ such that, for $F=\ran(\MG)$ and
the quotient space $Y := Y_F=\ell_p(\NN; X)/\ran(\MG)$ with natural quotient
map $q$ and $S:=\Phi_F(T)$, we have 
  \begin{equation}
    \label{eq:big-diagram}
    \begin{tikzcd}[column sep=huge]
      \ell_p(\NN; X) \arrow[r, "\THUT"] \arrow[d, "q"]
             & \ell_p(\NN; X)    \arrow[d, "q"] \\ %\arrow[dd, bend right=-40, "P"]
      Y  \arrow[r, "S"]
             & Y   \\
      X  \arrow[r, "T"] \arrow[u, "\iota"] \arrow[uu,bend left=40, "j"]
             & X \arrow[u, "\iota"]
   \end{tikzcd}
  \end{equation}
with $\iota := q \circ j:X\to Y$ being an isomorphism onto its range and $S$
being an isomorphism of $Y$.
This was achieved by Madani \cite{Madani} elaborating on some ideas from \cite{B-G}. 

\begin{theorem}[Madani~\cite{Madani}]\label{thm:Madani}
  Fix $\alpha>1$ and $p\in(1,\infty)$ such that $\alpha\ge 2^{1-1/p}$.
  Let $T\in\bdd(X)$ satisfy \eqref{eq:lower-bound} where $c>0$.
  Let $\MG:=I-\frac{\alpha}c\THUT R$. 

  \begin{enumerate}
  \item Then $F:=\ran(\MG)$ is closed and for $Y:=Y_F$ with quotient
    map $q$ the map $\iota:=q\circ j: X \to Y$ is an isomorphism onto
    its range and we have
    \begin{align*}
       \norm{\iota x}_Y \le \norm{x}_X \le \alpha\norm{\iota x}_Y,\quad x\in X.
    \end{align*}
  \item The operator $S:=\Phi_F(T):Y\to Y$ is an isomorphism with
    $\norm{S}\le\norm{T}$ and $\norm{S^{-1}}\le \frac\alpha{c}$ and
    the following diagram commutes.
  \[
    \begin{tikzcd}[column sep=large]
      Y \arrow[r,"S"] & Y  \\
      X \arrow[u,"\iota"] \arrow[r,"T"] & X \arrow[u,"\iota"]
    \end{tikzcd}
  \]
\item If $X$ is reflexive then $Y$ is reflexive.
\end{enumerate}
\end{theorem}

\begin{remark}\label{rem:Madani}\rm
  $(a)$ For fixed $\alpha>1$ we can choose $p\in(1,\infty)$ such that
  $2^{1-1/p}\le\alpha$. Hence $\norm{S^{-1}}$ can be made arbitrarily
  close to $\frac1c$.

  $(b)$ The construction respects other geometric properties of the
  Banach space $X$. Below we shall use that $Y$ is \UMD if $X$ is
  \UMD,  see,  e.g. \cite[Proposition~4.2.17]{fab4:band1}.
  But also a lot of other properties persist such as, e.g., uniform
  convexity or $B$-convexity or Pisier's property $(\alpha)$.  We do
  not elaborate on this further.
\end{remark}

\begin{proof}[Proof of Theorem~\ref{thm:Madani}] 
  \noindent{\bf Step 1: The construction of $F$.}  Let
  $\MG = I - \tfrac{\alpha}{c} \THUT R$. By the assumption on $\alpha$ and
  $p$, we have
  \begin{equation}    \label{eq:alpha-cond}
    (|a|+|b|)^p \le |\alpha a|^p + |\alpha b|^p\,\qquad a,b\in\C.
  \end{equation}
  By the triangle inequality and \eqref{eq:lower-bound} we infer
 \[
 \norm{ \MG x }_{\ell_p(\NN; X)} \ge (\alpha-1) \norm{x},
 \qquad x\in\ell_p(\NN; X),
 \]  
 and so $F = \ran(\MG)$ is closed by $\alpha>1$.

 \medskip\noindent{\bf Step 2: Lower bounds for the injection
   $\iota$.}  Recall we defined $Y = \ell_p(\NN; X) / F$ with quotient
 map $q$ as well as $\iota = q \circ j$.  This implies
 immediately $\norm{ \iota } \le 1$. On the other hand,
 \eqref{eq:alpha-cond} ensures
 \[
   \norm{ y}_X^p
   \le (\norm{ y- z}_X + \norm{z}_X)^p
   \le \norm{ \alpha(y-z)}_X^p + \norm{ \alpha z }_X^p
 \]
 Applying this inequality as well as \eqref{eq:lower-bound} 
 iteratively to a sequence $(z_k) \in \ell_p(\NN; X)$, we obtain
 \begin{align*}
   \norm{x}_X^p
   \le & \; \norm{ \alpha(x-z_1) }_X^p + \norm{ \alpha z_1 }_X^p \\
   \le & \; \norm{ \alpha(x-z_1) }_X^p + \norm{ \tfrac{\alpha}{c} T z_1 }_X^p \\
   \le & \; \norm{ \alpha(x-z_1) }_X^p
   + \norm{ \alpha ( \tfrac{\alpha}{c} T z_1 -z_2) }_X^p
   + \norm{ \alpha z_2}_X^p \\
   \le & \; \norm{ \alpha(x-z_1) }_X^p +
   \norm{ \alpha (\tfrac{\alpha}{c} T z_1 -z_2) }_X^p
   + \norm{ \tfrac{\alpha}{c} T z_2}_X^p \\
   \le & \; \norm{ \alpha(x-z_1) }_X^p
   + \norm{ \alpha (\tfrac{\alpha}{c} T z_1 -z_2) }_X^p
   + \norm{ \alpha (\tfrac{\alpha}{c} T z_2 - z_3)}_X^p + \norm{ \alpha z_3}
 \end{align*}
 etc. We infer that
 \[
   \norm{x}_X^p
   \le  \alpha^p \norm{ (x-z_1) }_X^p
   + \alpha^p \sum_{k=1}^\infty  \norm{\tfrac{\alpha}{c} T z_k -z_{k+1} }_X^p
   = \alpha^p \bignorm{ j(x) - (I - \tfrac{\alpha}{c} \THUT R) z }_{\ell_p(\NN; X)}.
 \]
 This implies $\norm{x}_X \le \alpha \norm{ \iota(x) }_Y$, so that
 $\iota: X \to Y$ is injective and has closed range. 

\medskip\noindent{\bf Step 3: Construction and invertibility of $S$.}
Since $\THUT$ and $\MG$ commute, $F=\ran(\MG)$ is invariant under $\THUT$,
so that
\[
  S=\Phi_F(T):Y\to Y, \qquad S([z])=[ \THUT z].
\]
is well-defined, cf. Lemma~\ref{lem:elem}.
For every $z \in \ell_p(\NN;X)$ we have
\[
  \MG z = z - \tfrac{\alpha}{c} \THUT R z \in \ran(\MG)=F.
\]
Hence in the quotient space $Y$,
\begin{equation}  \label{eq:inverse-prep}
  [z]=[\tfrac{\alpha}{c} \THUT R z].
\end{equation}
Now let $L [y] := \tfrac{\alpha}{c}\,  [R z]$ for any $z \in [y]$.
This is well-defined, since $R\, \MG = \MG R$ implies that
$F=\ran(\MG)$ is invariant under $R$.
%Indeed, let $z - z' \in J = \ran(\MG)$. Then $z = z' + \MG(w)$ so that
%$R z = R z' + \MG( R w)$, whence $[R z] = [R z']$.
Now we calculate
\[
S L [z] = \tfrac{\alpha}{c} \, S [ R z] = \tfrac{\alpha}{c}\,
[ \THUT  R z ] = \tfrac{\alpha}{c}\, [  \THUT R z]
  \overset{~\eqref{eq:inverse-prep}}{=\joinrel=} [z]
\]
so that $L$ is a right inverse of $S$. In the same way,
\[
L S[z] = L [ \THUT z ] = \tfrac{\alpha}{c}\, [\THUT R z]
\overset{~\eqref{eq:inverse-prep}}{=\joinrel=} [z].
\]
We observe that $\norm{ S } \le \norm{T}$ and that
$\norm{S^{-1}} = \norm{L} \le \tfrac{\alpha}{c}$.
\end{proof}

\begin{remark}
  Theorem~\ref{thm:Madani} extends to settings where
  $T_1 \in \cL(X, Y)$ admits lower bounds: let $X, Y$ be 
  Banach spaces, and $\norm{T_1 x}_Y \ge c \, \norm{x}_X$.  Then $T_1$
  admits a dilation to an invertible operator $S$. The idea is to
  consider the  space
  $Z := \ell_{p}(\ZZ_{\le 0}; X) \oplus \ell_p(\ZZ_{>0}; Y)$ and the
  bounded operator $T$ given by $T (z_n)_n = (w_{n})_n$ where for all
  $n\not=0$, $w_n = z_{n+1}$ and $w_0 = \tfrac1c T_1 z_1$. %\marginpar{$T$ STIMMT SO NiCHT!}
%   This
%   represents as double-infinite matrix
% \[
% \begin{array}{c|ccccccc}
%  & \cdots &  \scalebox{0.6}{$j={-}2$} & \scalebox{0.6}{$j={-}1$} & \scalebox{0.6}{$j=0$} & \scalebox{0.6}{$j=1$} & \cdots \\
% \hline
% \vdots                  & \ddots   & \ddots     &          &         &        &        \\
% \scalebox{0.6}{$j=-1$}  & \cell{}  & \cell{I_X} & \cell{0}    & \cell{}  & \cell{} & \cell{} \\
% \scalebox{0.6}{$j=0$}   & \cell{}  & \cell{}    & \cell{I_X} & \cell{0}  & \cell{} & \cell{} \\
% \scalebox{0.6}{$j=1$}   & \cell{}  & \cell{}    & \cell{}   & \cell{\tfrac1c T} & \cell{0} & \cell{} \\
% \scalebox{0.6}{$j=2$}   & \cell{}  & \cell{}    & \cell{}    & \cell{} & \cell{I_Y} & \ddots \\
% \scalebox{0.6}{$j=3$}   & \cell{}  & \cell{}    & \cell{}    & \cell{}  & \cell{} & \ddots \\
% \vdots                  &          &            &            &          &          & 
% \end{array}
% \]
% where we printed zeros on the main diagonal to suggest the shift
% structure.
  Now, $\norm{ T z }_Z \ge \norm{z}_Z$ and we can apply
  Theorem~\ref{thm:Madani}. As in the main result, if $X, Y$ are
  reflexive, so is $Z$, and the same is true for the \UMD property.
\end{remark}

\begin{remark}
  Let $P: \ell_p(\NN; X) \to X$ be the projection onto the first component, 
  so that $P\circ j = I_X$.  We observe that, in the above
  construction, no bounded projection $\pi: Y \to X$ with the property
  $P = \pi \circ q$ can exist (unless $X = \{0\}$, of
  course). Indeed, observe that $I_X = P \circ \MG \circ j$.
  Hence, assuming $P = \pi \circ q$ leads by $ q \circ \MG = 0$ to
  the contradiction $I_X = \pi \circ q \circ \MG \circ j = 0$.
\end{remark}

\noindent The following corollary sharpens some of  Madani's findings slightly.

\begin{corollary}[Madani~\cite{Madani}] \label{cor:Madani}
  Under the assumptions of Theorem~\ref{thm:Madani}, the map
  $ \Phi:=\Phi_F: \{ T \}' \to  \bdd(Y)$, $U\mapsto\Phi(U):=\Phi_F(U)$ is an
  algebra homomorphism satisfying
  \[
    \frac1\alpha \norm{U} \le \norm{ \Phi(U) } \le \norm{U},\qquad U\in\{T\}',
  \]
  for which the diagram
  \[
    \begin{tikzcd}[column sep=large]
      Y \arrow[r,"\Phi(U)"] & Y  \\
      X \arrow[u,"\iota"] \arrow[r,"U"] & X \arrow[u,"\iota"]
    \end{tikzcd}
  \]
  commutes,  i.e.  $\Phi(U)\iota = \iota U$. More precisely, for all $U\in\{T\}'$ and
  $x\in X$, we have
\begin{equation}  \label{eq:norm-inequality}
  \frac{1}{\alpha}\norm{ Ux }_X \le \norm{ \Phi(U)\iota x }_Y    \le \norm{ Ux }_X.
\end{equation}
\end{corollary}
\begin{proof}
  Let $U \in \{T\}'$, so that $U$ and $T$ commute. Then also $\UHUT$ and $\THUT$
  commute, and hence $\UHUT$ and $\MG$ commute, which implies that $\UHUT$
  leaves $F=\ran(\MG)$ invariant. Thus $\Phi(U):=\Phi_F(U)$ is well-defined,
  cp. Lemma~\ref{lem:elem}.
%  Therefore $\UHUT$ induces a bounded operator on the quotient
%\[
%  \Phi(U):Y\to Y, \qquad \Phi(U)[z]=[\UHUT z].
%\]

  Clearly, $\Phi(I_X) = I_Y$, and by definition $\Phi$ is linear.
Now let $U, V \in \{T\}'$. Then
\[
  \Phi(UV)[z] = [\widehat{UV} z],
\qquad
  \Phi(U)\Phi(V)[z] = [\widehat{U}\,\widehat{V} z].
\]
Since $\widehat{UV} = \widehat{U}\,\widehat{V}$ on $\ell_p(\NN;X)$,
the multiplicativity follows. Finally, observe
\[
  \Phi(U)\iota x = \Phi(U) [ j(x) ] = [ \UHUT j(x) ] = [ j(Ux) ] = \iota(Ux)
\]
Using the norm equivalence on $\iota(X)$ from Theorem~\ref{thm:Madani} and
the previous identity we obtain, for all $x\in X$,
\[
\frac{1}{\alpha}\norm{ Ux }_X \le \norm{ \Phi(U)\iota x }_Y
= \norm{ \iota(Ux) }_Y   \le \norm{ Ux }_X.
\]
This implies
$\frac1\alpha \norm{ U }_{\bdd(X)} \le \norm{ \Phi(U) }_{\bdd(Y)}$ on the one
hand, while
$\norm{ \Phi(U) [z] } = \norm{ [ \UHUT z ] } \le \norm{U}_{\bdd(X)}
\norm{ [z] }$ implies $\norm{ \Phi(U) }_{\bdd(Y)} \le \norm{ U }_{\bdd(X)}$.
\end{proof}

\section{Results and examples concerning semigroups}\label{sec:semigroups}

\begin{proof}[Proof of Proposition~\ref{prop:sg-lower-bound}]
Fixing $\alpha>1$ we choose $p\in(1,\infty)$ such that
$\alpha\ge 2^{1-1/p}$.  Applying Theorem~\ref{thm:Madani} and
Corollary~\ref{cor:Madani} to $T=T(t_0)$ we obtain $Y$, $\iota:X\to Y$
and $\Phi:\{T(t_0)\}'\to\bdd(Y)$.

\noindent We define $S(t):=\Phi(T(t))$ for $t\ge0$, which is a
$C_0$-semigroup in $Y$ by Lemma~\ref{lem:elem}, and have that
$S(-t_0):Y\to Y$ is bijective with
$\norm{S(t_0)^{-1}} \le \frac{\alpha}{c}$. But %it is well known that
a $C_0$-semigroup extends to a $C_0$-group if one of the semigroup
operators is invertible, see
e.g. \cite[Theorem~16.3.6]{Hille-Philips}. Hence we obtain a
$C_0$-group $(S(t))_{t\in\R}$ in $Y$ extending the given semigroup.

\noindent We now examine the growth bound of $(S(-t))_{t\ge0}$.  We have,
for $0 \le \delta \le t_0$,
\[
  \norm{ S(-nt_0 - \delta) } \le K e^{\ln(\frac{\alpha}{c}) n}
\]
where $K = \max\{ \norm{S(-\delta)}: 0 \le \delta \le t_0 \}$. A standard
calculation thus yields
\[
m \, e^{\nu t} \, \norm{x} \le \norm{ S(t) x } %\le M \,e^{\omega t} \, \norm{x},
\quad x\in X, t>0,
\]
for $\nu = \tfrac1{t_0}\ln(\frac{c}{\alpha})$ and some $m>0$. Now by
\eqref{eq:norm-inequality} we obtain, for any $x\in X$,
\[
  \norm{ T(t) x }_X  \ge \norm{ \Phi( T(t)) \iota x }_Y = \norm{ S(t) \iota x } \ge   m \, e^{\nu t} \, \norm{\iota x}  \ge \frac{m}{\alpha} \, e^{\nu t} \, \norm{x}. \qedhere
\]
\end{proof}

\begin{example}\label{example:sg-growth}
  The following example shows that strong continuity of a semigroup on
  $(0, \infty)$ does not imply growth bounds as $t\to 0+$. In particular,
  a function $\gamma:(0,\infty)\to[0,\infty)$ satisfying
  $\gamma(t+s)\le\gamma(t)\gamma(s)$ for all $s,t>0$ can grow arbitrarily
  fast for $t\to0+$.
  %and so does not guarantee strong continuity at zero.
  
  \noindent Consider $X = L^p(\RR_+) \times L^p(\RR_+)$ and let
  \[
    A \left(\begin{array}{c} f_1 \\ f_2 \end{array} \right)
    = 
    \left(\begin{array}{cc} \phi(x) & -e^x  \\ 0 & \phi(x) \end{array} \right)\,
    \left(\begin{array}{c} f_1 \\ f_2 \end{array} \right)
  \]
  where $\phi: \RR_+ \to \RR_+$ is a function of class $C^1$. Let us
  assume additionally $\phi(0)=\phi'(0)=0$ and that
  $\phi': \RR_+ \to \RR_+$ is strictly increasing with $\phi'(x)\to\infty$ for
  $x\to\infty$ and $\phi'(\log x)\le x$ for $x>1$.  A direct calculation
  shows
  \[
    \exp(tA) =  e^{-t\, \phi(x)}   \left(\begin{array}{rr} 1 &  t e^x  \\ 0 & 1 \end{array} \right),
  \]
  and it follows that
  $\norm{ \exp(tA) } \simeq \sup_{x>0} \; e^{-t\, \phi(x)} \max\{1, t e^{x} \}$.
  Now  $\sup_{x>0} e^{-t\, \phi(x)} = 1$ and for $t\in(0,1)$ we thus have
  \[
  \sup_{x>0} \; e^{-t\, \phi(x)} \max\{1, t e^{x} \}
  =\max\{1, \sup\{ t e^{x-t\phi(x)}:\,x\ge\log(1/t)\}\} 
  \]
  and
  \[
  \sup_{x\ge\log(1/t)}  \; t \exp(x -t\, \phi(x) )
  = t \exp\Big( t \sup_{x\ge\log(1/t)} \,\big(\frac{x}t - \phi(x) \big) \Big)
  \]
  Let $\phi^*(s) := \sup_{x>0} \; (s x - \phi(x))$ denote the convex
  conjugate, or Young dual function of $\phi$ and observe that, for $s>1$ and
  under our assumptions, the $\sup$ is attained at $x_0$ with
  $\phi'(x_0)=s\ge\phi'(\log s)$ which means $x_0\ge\log s$. Hence 
\[
  \sup_{x\ge\log(1/t)} \; \left(\frac{x}{t} - \phi(x)\right) =
  \phi^*\!\left(\frac{1}{t}\right),\qquad t\in(0,1).
\]
As a consequence, the operator norm of the semigroup can grow
arbitrarily fast as $t \to 0+$.  For example, if
$\phi(x) \sim x \log^{\circ k}(x)$, then
$\phi^*(s) \sim \exp^{\circ (k+1)}(s)$ as $s \to \infty$.
\end{example}

\bigskip

\begin{example}\label{example:BG}
  The following example shows there exist semigroups on non-Hilbertian
  Banach spaces that do allow a lower estimate, but no left inverse
  semigroup.
  
  Recall that every infinite dimensional, non-Hilbert Banach space
  contains a closed non-complemented subspace, see \cite{LT}.  So let
  $Z$ a \UMD Banach space, and $W \subseteq Z$ a non-complemented,
  closed subspace. We let $1<p<\infty$,
  \[
    X = L^p(\RR_-; Z) \oplus L^p(\RR_+; W),
  \]
  and $(T(t))_{t\ge0}$ the left shift semigroup on $X$. It is isometric on $X$
  since we shift from $W$ into $Z$ and both carry the same norm.
  Fix $t_0>0$. We let\\[1ex]
  \[
    \begin{array}{llll}
    j_0: & \; \left\{
           \begin{array}{lcl}
             W & \to & X\\
             w & \mapsto & t_0^{-\sfrac1p} \dsp\eins_{(0, t_0)} \otimes w
           \end{array}\right., &
    j_1: & \; \left\{
           \begin{array}{lcl}
             Z & \to & X\\
             z & \mapsto & t_0^{-\sfrac1p} \dsp\eins_{(-t_0, 0)} \otimes z
           \end{array}\right.,\\[5ex] 
    M_0: & \; \left\{
           \begin{array}{lcl}
             X & \to & W\\
             f & \mapsto & t_0^{\sfrac1p-1} \dsp\int_0^{t_0} f(t)\,dt
           \end{array}\right.. &
%    M_{-1}: & \; \left\{
%           \begin{array}{lcl}
%             X & \to & Z\\
%             f & \mapsto & t_0^{\sfrac1p-1} \dsp\int_{-t_0}^0 f(t)\,dt.
%           \end{array}\right.\\
    \end{array}
  \]
  Observe that $M_0j_0=I_W$ and that $j_1I_W = T(t_0) j_0$. %We will use this twice below.
  Now assume towards a contradiction that $S(t_0)\in \bdd(X)$ is a left inverse
  to $T(t_0)$. Then
  \[
    P := M_0 S(t_0) j_1\in\bdd(Z)
  \]
  satisfies
  \[
    P^2 = (M_0 S(t_0) j_1)^2
        =  M_0 S(t_0) \; T(t_0) j_0  \; M_0 S(t_0) j_1
        =  M_0  j_0 M_0 S(t_0) j_1
        = P
  \]
  so that $P$ is a projection satisfying
  $\ran(P) \subseteq \ran M_0 \subseteq W$. Finally,  for $w \in W$,
  \[
    P w = M_0 S(t_0) \;j_1 \; w = M_0 S(t_0) \; T(t_0) j_0 \; w = M_0 j_0 w = w,
  \]
  which contradicts $W$ being not complemented.  This example shows
  that the obstruction to a generalization of Theorem~\ref{thm:BG} lies in
  Banach space geometry.
  %and does not rely, for example, on any specific PDE structure.
\end{example}

\section{Proof of the main result}\label{sec:proof-main}

As mentioned in the introduction, the argument relies on Madani's
dilation that we reviewed in Section~\ref{sec:madani} and the
functional calculus for groups obtained by Haase on strips
using transference methods. For a nontrivial complex domain $\calO$ we
denote by $H^\infty_1(\calO)$ the space of functions $f\in H^\infty(\calO)$
for which $z\mapsto zf'(z)$ is bounded on $\calO$, and we define
$\calE(\calO):=\{f\in H^\infty_1(\calO): \sup_{z\in\calO}(1+|z|)^2|f(z)|<\infty\}$.
For $\al\in\R$ and $\beta>0$
we define open half-planes $\HP_\al$ and and open strips $\strip_\beta$ by
\[
\HP_\al = \{ z\in\C: \; \Re(z)> \al \},\qquad
\strip_\beta = \{ z\in\C: |\Re(z)|<\beta \}.
\]
The result we shall use states that, for a group $(U(t))_{t\in\R}$ in a \UMD
space with $\norm{ U(t) }\le M e^{\om|t|}$, $t\in\R$, the generator has a
bounded $H^\infty_1(\strip_\eta)$-calculus for any $\eta>\om$. The precise
construction of the functional calculus is reviewed in the proof.

\begin{proof}[Proof of Theorem~\ref{thm:HP-transfer} ]
  Replacing $A$ by $A+\omega+\del$ with $\del>0$, if necessary, we may
  suppose $\om < 0$ for the growth bound, i.e. exponential stability
  of the semigroup, and have to show that $A$ has a bounded
  $H^\infty(\Sect_\si)$-calculus for all $\si>\pihalbe$.

  \medskip \noindent Part~\ref{item:sg-lower-b} of
  Proposition~\ref{prop:sg-lower-bound} and subsequent remarks ensure
  that the space $Y$ we obtained by application of
  Theorem~\ref{thm:Madani} and Corollary~\ref{cor:Madani} to
  $T:=T(t_0)$ is again a \UMD space and that the $C_0$-semigroup
  $\big(\Phi(T(t))\big)_{t\ge0}$ extends to a $C_0$-group
  $(S(t))_{t\in\R}$ in $Y$.  These operators satisfy
  \begin{equation}\label{eq:intertwining-semigroup}
    S(t) \iota = \iota T(t) \qquad (t\ge 0).
  \end{equation}
  Let ${-}A$ denote the generator of $(T(t))$ and ${-}A_Y$ the
  generator of $(S(t))$ (and not $i\, A_Y$ which is customary if one only
  deals with groups).  From Proposition~\ref{prop:sg-lower-bound} we
  know that
  \[
    m \, e^{\nu t }\, \norm{y} \le \norm{ S(t) y } \le M\, e^{\om t}\, \norm{y} \qquad (y \in Y, t\ge 0)
  \]
  for some $\nu < \om < 0$, and some $m, M>0$. In particular, the
  group type of $S$ (see \cite[p.302]{Haase:buch}) is $\le{-}\nu$ and the
  spectrum of $A_Y$ lies in the strip
  \[
    \strip_{-\nu} = \{ z: \quad |\Re(z)| \le -\nu \}.
  \]
  Fix $\eta> -\nu$. Let us write
  $\calE( \strip_\eta) := \{ f \in H^\infty_1(\strip_\eta): \;
  |f(z)| = O((1+|z|)^{-2}) \}$.
  %% Self-note: our group generator is "strip type" but not "strong
  %% strip type": the resolvent is bounded in right half planes, but
  %% has, a priori, no uniform decay (espescially if Im(z) \to \pm
  %% \infty). Thus "-2" as exponent.
  Then by Haase's generalisation
  \cite[Theorem~3.6]{Haase:transference} of the Hieber--Pr\"uss
  theorem \cite{Hieber-Pruess}, $A_Y$ admits a bounded
  $H_1^\infty(\strip_\eta)$-functional calculus, that is, there
  exists some $C>0$ such that, for   $f \in \calE( \strip_\eta )$ and
  $\sigma\in(-\nu,\eta)$,
  the absolutely convergent integral
  \begin{equation}    \label{eq:elem-fc-for-AY}
    f(A_Y) = \frac{1}{2\pi i} \int_{\partial \strip_\sigma} f(z) R(z, A_Y)\,dz
  \end{equation}
  satisfies
  $\norm{ f(A_Y) } \le C \norm{f}_{H_1^\infty(\strip_\eta)}$, see
  \cite[Cor. 3.4]{Haase:transference}.  Now consider
  $f \in \calE( \HP_{-\eta})$. %where $\HP_r$ denotes the halfplane.
  By holomorphy, the right hand boundary of the curve integral in
  \eqref{eq:elem-fc-for-AY} can be shifted to the right, i.e. from
  $\Re(z) = \sigma$ to $\Re(z) = r$, for any $r> \eta$.  Letting
  $r\to {+}\infty$, uniform boundedness of the resolvents and the
  decay of $f$ make this curve integral vanish, so that
   \begin{equation}    \label{eq:halpfplane-fc-for-AY}
    f(A_Y) = \frac{-1}{2\pi i} \int_{-\sigma+i\RR} f(z) R(z, A_Y)\,dz
  \end{equation} 
  with absolutely converging integrals, which gives us a functional calculus
  on $\calE(\HP_{-\eta})$ for $A_Y$ with the estimate
  $\norm{ f(A_Y) } \le C \norm{f}_{H_1^\infty(\HP_{-\eta})}$.
  The absolutely converging integral
  \begin{equation}    \label{eq:halpfplane-fc-for-A}
    f(A) := \frac{1}{2\pi i} \int_{\Re(z) = -\eta} f(z)\, R(z, A)\,dz
  \end{equation}  
  defines a bounded operator $f(A)\in\bdd(X)$ for all
  $f \in \calE( \HP_{-\eta})$.
  By the properties obtained in Corollary~\ref{cor:Madani}, we may apply
  Laplace transform to \eqref{eq:intertwining-semigroup} and obtain
  \[
     R(\lambda, A_Y) = \Phi( R(\lambda, A) ) \qquad (\Re(\lambda) < -\omega).
  \]
  Furthermore this entails $f(A_Y) = \Phi( f(A) )$.
  Using Corollary~\ref{cor:Madani} again we conclude
  \[
  \norm{ f(A) }_{\bdd(x)}\le\al\norm{ \Phi( f(A) ) }_{\bdd(Y)}
  = \al\norm{ f(A_Y) }_{\bdd(Y)}
  \le \al C \norm{ f }_{H^\infty_1(\calE(\HP_{-\eta}))}.    
  \]
  We also have
  \[
    \Phi( (fg)(A) ) = (fg)(A_Y) = f(A_Y) g(A_Y) = \Phi( f(A) ) \Phi( g(A) ) = \Phi( f(A) g(A) ),
  \]
  so that the assignement $f \mapsto f(A)$ is multiplicative. Let
  \[
     \varrho(z) = \frac{z}{(1+\eta'+z)^2}\qquad \varrho_n(z) = \varrho(nz) - \varrho(\frac{z}n).
  \]
  We set $f_n(z) := \varrho_n(z)^2 f(z)$.  Then
  $f_n \in \calE( \HP_{-\eta} )$, $f_n \to f$ pointwise, boundedly, and 
  \[
    \norm{ f_n(A) } \le C \norm{ f_n }_{H^\infty_1(\HP_{-\eta})} \le C' \norm{ f }_{H^\infty_1(\HP_{-\eta})}.
  \]
  By the convergence lemma, see e.g. \cite[Chapter~5]{Haase:buch},
  $f(A)$ is bounded on $X$ and $f_n(A) \to f(A)$ strongly, i.e.  $A$
  has a $H^\infty_1(\HP_{-\eta})$ functional calculus.
  Shifting both the operator and the function class,
  we obtain that $A+\eta$ has a bounded $H^\infty_1(\HP_0)$ functional
  calculus and it is well known, see
  e.g. \cite[Lemma~4.5]{Haase:transference}, that $A+\eta$ then has a
  bounded $H^\infty(\Sect_\si)$ functional calculus for all
  $\si>\pihalbe$.

  \medskip\noindent Since $A$ is sectorial of angle $\pihalbe$ and
  $0 \in \varrho(A)$, \cite[Prop.~6.10]{KKW} shows that $A$ has a
  bounded $H^\infty(\Sect_\si)$ functional calculus for all $\si>\pihalbe$.
\end{proof}

An inspection of the proof shows that a little more can be said.  For
$a, r \in \RR$ and $ \pihalbe < \si < \pi$ we consider the following
sets
\begin{equation}   \label{eq:K-shape}
  \KR_{\si, a, r} := ( a + \Sect_\si)  \, \cup \; \{ z: \; \Re(z)> r \}
   = ( a + \Sect_\si )\cup\HP_r.
\end{equation}
For such sets the following folklore lemma becomes relevant. 

\begin{lemma}\label{lem:folklore} 
  Let $\eta >0$. Then for any $\eps > 0$, $a\in\R$, and $\si>\pihalbe$
  there exists $C_{\eps, a, \si}>0$ such that
  \[
    \norm{f}_{H_1^\infty( \HP_{-\eta} )}   \le C_{\eps, \si}\, \norm{ f }_{ H^\infty( \KR_{\si, a, -\eta-\eps}  ) }
  \]
\end{lemma}
\begin{proof}
  \begin{figure}
    \centering
  \begin{tikzpicture}[scale=1]
% Parameters
\def\A{(1,0)}
\def\angle{110}

% --- Shaded region ---
% 0) K region.
\draw[white, fill=gray!10]
({1+6*cos(\angle+10)}, {-6*sin(\angle+10) }) --
(6, {-6*sin(\angle+10)} ) --
(6, {6*sin(\angle+10) }) --
({1+6*cos(\angle+10)}, {6*sin(\angle+10) }) --
(-1, {-2 *tan(\angle+10)} )  --
(-1, {2 *tan(\angle+10)} )  --
cycle  ;
\draw (5, 4) node {$\KR_{\si, a, -\eta-\eps}$};

% 1) Disk centered at 
\begin{scope}
\clip (-0.5, -6) rectangle (6, 6); % Define a clipping box to the right of x = -1
\fill[gray!30] \A circle ({3/2*tan(\angle)-0.3});
\end{scope}

% --- Point A ---
\fill \A circle (2pt) ;

% ---  Point z
\fill (-0.4,5) circle (1pt) node[right] {$z$};
\draw[->] (-0.4+0.7,5) arc(0:360:0.7) ; 

% ---  other Point z
\fill (-0.4,-1) circle (1pt) node[right] {$z$};
\draw[->] (-0.4, -1) ++(60:0.3) arc(60:60+359.9:0.3);

% --- \sigma' Rays from A at \angle ---
\draw[dotted] \A -- ++({\angle}:5);
\draw[dotted] \A -- ++({-\angle}:5);
\draw (1.7,0) arc (0:\angle:0.7) ;
\draw (1.3, 0.3) node {$\si'$};
% --- \sigma Rays from A at \angle+10 ---
\draw \A -- ++({\angle+10}:6);
\draw \A -- ++({-\angle-10}:6);
\draw (1.7,0) arc (0:{-\angle-10}:0.7) ;
\draw (1.2, -0.3) node {$\si$};

% --- Vertical line x = -eta ---
\draw[dashed] (-0.5,-5) -- (-0.5,6) ;
\draw (-0.3, -0.2) node {${-}\eta$};

% --- Vertical line x = -eta-eps ---
\draw (-1,-5) -- (-1,6) ;
% --- K
\draw (2, 2) node {$K$};
% Axes
\draw[thick, ->] (-1.5,0) -- (7,0) node[below, xshift=-3pt] {$x$};
\draw[thick, ->] (0,-5) -- (0,6.3) node[below, xshift=4pt] {$y$};
\end{tikzpicture}

\caption{The region $\KR_{\si, a, -\eta-\eps}$, two cases for a point
  $z$ in the half-plane $\HP_{-\eta}$ and respective integration
  paths.}
    \label{fig:folklore}
  \end{figure}
  We use Cauchy's formula and combine the estimates from two
  repesentations.  We limit our attention to the case $a > -\eta$, since
  otherwise $\KR_{\si, a, -\eta} = a + \Sect_\si$.  

  \medskip\noindent For $z\in \HP_{-\eta}$ we distinguish 2 cases: we
  fix $\si'\in(\pihalbe,\si)$ and define the bounded zone $K$ by
  $\Re(z) > {-}\eta$ and
  \[
    |z-a|\le \max\big\{2a, \frac{\eta+a}{|\cos(\si')|}\big\}.
  \]
  For $z\in K$ we use 
\[
  z\,f'(z)=\frac1{2\pi i}\int\limits_{|\zeta-z|=\eps/2}\frac{z
    \, f(\zeta)}{(\zeta-z)^2}\,d\zeta,
\]
leading to
\[% \begin{align*}
|z \, f'(z)| \le  \; \frac1{2\pi}\,\pi\eps \frac{4|z|}{\eps^2}\,
\norm{ f }_{ H^\infty( \KR_{\si, a, -\eta-\eps}  ) }
\le  \; \frac{4M}\eps\,\norm{ f }_{ H^\infty( \KR_{\si, a, -\eta-\eps}  ) }.
\]% \end{align*}

For $z$ in the complement $\HP_{-\eta} \setminus K$ we use the fact that
$|\arg(z)| \le \sigma' < \sigma$, which allows a suitable choice of
$\del \in (0,  \sin(\sigma-\sigma') )$ to write
\[
z\, f'(z)=\frac1{2\pi i}\int\limits_{|\zeta-z|=\del|z-a|}
\frac{z f(\zeta)}{(\zeta-z)^2}\,d\zeta.
\]
Using $|z-a|\ge 2a$ we have
\[
  \frac12|z|  \le |z-a|\le \frac32 |z|,
\]
which leads to
\[% \begin{align*}
|z f'(z)| \le \frac{2\pi\del|z-a||z|}{2\pi\del^2|z-a|^2}\,
\norm{ f }_{ H^\infty( \KR_{\si, a, -\eta-\eps}  ) } 
\le \frac{6}{\del}\,\norm{ f }_{ H^\infty( \KR_{\si, a, -\eta-\eps}  ) }.\qedhere
\]% \end{align*}
\end{proof}

\noindent The following is a corollary to the proof of Theorem~\ref{thm:HP-transfer}.

\begin{corollary}\label{cor:K-fc}
  Under the assumptions of Theorem~\ref{thm:HP-transfer} the operator $A$
  has a bounded $H^\infty(\KR_{\si,a,-\theta})$-calculus for any $\theta>\om$,
  $\si>\pihalbe$, and $a\in\R$.
\end{corollary}

Here, for $f\in H^\infty(\KR_{\si,a,-\theta})$ satisfying in addition
$|f(z)|=O((1+|z|)^{-\del})$ for some $\del>0$, the operator $A$ is defined by
the absolutely convergent integral
\[
 f(A)=\frac1{2\pi i}\int_{\partial\KR_{\si',a,-\theta'}} f(\la )R(\la,A)\,d\la,
\]
where $\si'\in(\pihalbe,\si)$ and $\theta'\in(\om,\theta)$. This defines a
functional calculus $f\mapsto f(A)$, and boundedness of an
$H^\infty(\KR_{\si,a,-\theta})$-calculus means that there exists $C>0$ with
$\norm{f(A)}_{\bdd(X)}\le C\norm{f}_{H^\infty(\KR_{\si,a,-\theta})}$ for all such $f$.
Again, we then have a unique extension to a bounded algebra homomorphism
$H^\infty(\KR_{\si,a,-\theta}) \to \bdd(X)$ via the convergence lemma.

\begin{proof}
Again, we resort to the case $\om<0$ and have to show the assertion for
$\theta=0$.
Since we saw in the proof that $A$ has a bounded
$H^\infty_1(\HP_{-\eta})$ functional calculus for all $\eta>-\nu$,
Lemma~\ref{lem:folklore} implies that it also has a bounded
$H^\infty( \KR_{\si, a, -\eta} )$ functional calculus for 
all $\eta>-\nu$, all $\si>\pihalbe$ and all $a \in \RR$.
The freedom of $a \in \RR$ allows
shifting, so that $A+\eta$ has a bounded $H^\infty( \KR_{\si, a, 0} )$
functional calculus for all $a\in \RR$ and all $\si>\pihalbe$. Now we
can mimic the main idea of the proof of \cite[Prop.~6.10]{KKW}. Let
$\Gamma$ be the boundary curve of $\KR_{\si', a, -\theta'}$ where
$\si'\in(\pihalbe,\si)$ and $\theta'\in(\om,0)$. Then for bounded
holomorphic functions $f\in\KR_{\si,a,0}$ with
$|f(z)|=O((1+|z|)^{-\del}$ for $|z|\to\infty$ and some $\del>0$,
%say, quadratic decay when $|z| \to \infty$,
we have three convergent integrals,
\begin{align*}
 \tfrac{1}{2\pi i} \int_\Ga f(\la)  R(\la, A)\,d\la
 = & \;   \tfrac{1}{2\pi i} \int_\Ga f(\la) R(\la, A_\eta) \,d\la +    \tfrac{\eta}{2\pi i} \int_\Ga  f(\la) R(\la, A+\eta) R(\la, A)) \,d\la\\
 = & \; f(A_\eta) +  \tfrac{\eta}{2\pi i} \int_\Ga  f(\la) R(\la, A+\eta) R(\la, A)) \,d\la.
\end{align*}
but the right hand side allows an upper estimate against
$\norm{f}_{\infty}$. The standard approximation argument then gives a
bounded $H^\infty( \KR_{\si, a, 0})$-functional calculus for $A$.
\end{proof}

\end{document}